\documentclass[12pt, a4paper]{article}

\usepackage{amsmath,amssymb,amsthm}
\usepackage{bbm,color,dcolumn,enumerate}
\usepackage{fancybox,fancyhdr,graphicx}
\usepackage{multirow,psfrag,pgfplots,relsize,subcaption,times}
\usepackage{tikz,tikz-3dplot}
\usepackage[affil-it]{authblk}
\usepackage[sort,numbers]{natbib}
\usepackage{setspace} 
\usepackage{hyperref}

\usepackage[english]{babel}

\usepackage{verbatim}

\usepackage{a4wide}

\usepackage{tikz}
\usepackage{graphicx, caption}
\usepackage{subcaption}
\usepackage{todonotes}

\newcommand{\prob}{\mathbb{P}}
\newcommand{\Prob}[1]{\prob\left(#1\right)}

\newcommand{\C}{{\rm\bf C}}

\allowdisplaybreaks[1]

\numberwithin{equation}{section}
\newtheorem{thm}{Theorem}
\newtheorem{lem}[thm]{Lemma}

\newtheorem{prop}[thm]{Proposition}
\theoremstyle{definition}
\newtheorem{remark}[thm]{Remark}
\numberwithin{thm}{section}

\newcommand{\be}{ \begin{equation}}
\newcommand{\ee}{\end{equation}}
\newcommand{\ben}{ \begin{equation*}}
\newcommand{\een}{\end{equation*}}
\newcommand{\nn}{\nonumber}
\newcommand{\oF}{\overline{F}}

\def\P{{\mathbb P}}

\def\Fb{{\overline{F}}}

\newcommand{\cutoff}{\gamma_n}

\newcommand{\pr}{\rightarrow}

\newcommand{\ba}{\begin{array}}
\newcommand{\ea}{\end{array}}

\newcommand{\eps}{\varepsilon}

\newcommand{\il}{\int\limits}
\newenvironment{inspring}[1]%
{\begin{list}{}{\setlength{\rightmargin}{0cm}
                \setlength{\listparindent}{0cm}
                \settowidth{\labelwidth}{\mbox{#1}}
                \setlength{\leftmargin}{1.1\labelwidth}
                \setlength{\labelsep}{.1\labelwidth}}}%
{\end{list}}
\newcommand{\ITEM}[1]{\item[#1\hfill]}
\newcommand{\bi}[1]{\begin{inspring}{#1}}
\newcommand{\ei}{\end{inspring}}

\newcommand{\dil}{\displaystyle \int\limits}

\newcommand{\bfc}{{\bf c}}

\newcommand{\bfd}{{\bf d}}

\newcommand{\bft}{{\bf t}}

\newcommand{\bfu}{{\bf u}}
\newcommand{\bfw}{{\bf w}}

\newcommand{\beq}{\begin{equation}}
\newcommand{\eq}{\end{equation}}
\catcode`\@=11

\font\tenmsa=msam10 \font\sevenmsa=msam7 \font\fivemsa=msam5
\font\tenmsb=msbm10 \font\sevenmsb=msbm7 \font\fivemsb=msbm5
\newfam\msafam
\newfam\msbfam
\textfont\msafam=\tenmsa  \scriptfont\msafam=\sevenmsa
  \scriptscriptfont\msafam=\fivemsa
\textfont\msbfam=\tenmsb  \scriptfont\msbfam=\sevenmsb
  \scriptscriptfont\msbfam=\fivemsb

\def\Bbb{\ifmmode\let\next\Bbb@\else
 \def\next{\errmessage{Use \string\Bbb\space only in math mode}}\fi\next}
\def\Bbb@#1{{\Bbb@@{#1}}}
\def\Bbb@@#1{\fam\msbfam#1}
\newcommand{\dR}{{\Bbb R}}

\newcommand{\e}{{\rm e}}

\begin{document}
\title{Counting cliques and cycles in scale-free \\ inhomogeneous random graphs}
\author[]{A.J.E.M. Janssen}
\author[]{Johan S.H. van Leeuwaarden}
\author[]{Seva Shneer}
\affil[]{Eindhoven University of Technology and Heriot-Watt University}
\renewcommand\Authands{ and }

\maketitle
 \abstract{Scale-free networks contain many small cliques and cycles. We
model such networks as inhomogeneous random graphs with regularly varying infinite-variance weights. For these models, the number of cliques and cycles have exact integral expressions amenable to asymptotic analysis. We obtain various asymptotic descriptions for how the average number of cliques and cycles, of any size, grow with the network size. For the cycle asymptotics we invoke the theory of circulant matrices. 
 }

\section {Introduction}\label{intro}


We study the number of cliques and cycles in scale-free random graphs with power-law degree distributions that have an infinite second moment. Such random graphs contain many small subgraphs \cite{bianconi2005loops, bianconi2006emergence,hofstad2017d}. 
Cliques are subsets of vertices that together form a complete graph and cycles are closed paths that visit each vertex only once. 

We employ the rank-1 inhomogeneous random graph or hidden-variable model \cite{boguna2003,park2004,bollobas2007,britton2006,norros2006,chung2002}, which generates power-law random graphs, to derive asymptotic expressions for the average number of cliques and cycles, of arbitrary size. This model constructs simple graphs with soft constraints on the degree sequence~\cite{chung2002,boguna2003}. The graph consists of $n$ vertices with weights $(h_1,h_2,\ldots,h_n)$. These weights are an i.i.d.\ sample from the power-law distribution
\begin{equation}\label{eq:pl}
\oF(h) = \Prob{H>h}=l(h)h^{1-\tau}, \quad  h\geq 1,
\end{equation}
for some slowly-varying function $l(h)$ and power-law exponent $\tau\in(2,3)$. 
We denote the average value of the weights by $\mu$. Then, every pair of vertices with weights $(h_i,h_j)$ is connected with probability 
\begin{equation}\label{eq:phh}
		p(h_i,h_j)=\min\left(\frac{h_i h_j}{\mu n},1\right),
\end{equation}
which is the Chung-Lu version of the rank-1 inhomogeneous random graph~\cite{chung2002}. This connection probability ensures that the degree of a vertex with weight $h$ will be close to $h$~\cite{boguna2003}, and that the probability $p(h_i,h_j)$ remains in the interval $[0,1]$. 


An alternative way to guarantee $p(h_i,h_j)\in[0,1]$
is to assume that the support of the weight distribution is restricted to $[0,\sqrt{n \mu}]$, so that the product $h_ih_j$  never exceeds $n \mu$, making the minimum operator superfluous. 
Banning degrees larger than the $\sqrt{n \mu}$ (also called the structural cutoff), however, violates the reality of scale-free networks in which hubs of expected degree $(n\mu)^{1/(\tau-1)}\gg \sqrt{n \mu} $ occur.  We therefore choose to work with \eqref{eq:phh} and \eqref{eq:pl},
putting no further restrictions on the range of the weights (and hence degrees).
This creates degree correlations, also observed in real scale-free networks, and 
an average connectivity and clustering coefficient that depend on the vertex weight/degree \cite{stegehuis2017,stegehuis2017b}	.

The goal is then to obtain sharp asymptotic estimates for 
the average number of $k$-cliques $A_k(n)$, which can be expressed as 
\begin{equation} \label{v1:1}
A_k(n):={n \choose k}\P(K_k) 
\end{equation}
with $\P(K_k)$ the probability that $k$ arbitrary vertices together form a $k$-clique. Similarly, the average number of $k$-cycles $C_k(n)$ satisfies
\be \label{defcycle}
C_k(n) := \dfrac{k!}{2k}{n \choose k} \P(C_k) 
\ee
with $\P(C_k)$ the probability that $k$ arbitrary vertices together form a $k$-cycle.
The combinatorial factor $\frac{k!}{2k}{n\choose k}$ is built from the usual factor ${n\choose k}$, due to choosing the set of $k$ out of $n$ vertices, and the factor $k!$, being the number of permutations of the chosen set. This has to be divided by ${k \choose 1}$, accounting for choosing a starting point of the cycle, and by 2, accounting for choosing the immaterial cycle's orientation.
\vspace{.2cm}

\noindent{\bf Imposing a cutoff}.
Bianconi and Marsili \cite{bianconi2005loops,bianconi2006emergence} start from an  exact integral for $\P(K_k)$, and impose the cutoff $\sqrt{n\mu}$, so that the integral can be transformed into a form amenable to asymptotic analysis through the saddle point method. We now repeat some of their arguments, and provide an alternative derivation that does not rely on the saddle point method. With the cutoff $\sqrt{n \mu}$ the probability $\P(K_k)$ can be expressed in terms of the $k$-fold integral
\begin{equation}\label{v1}
\P(K_k)=\int_{1}^{\sqrt{n \mu}}\cdots \int_{1}^{\sqrt{n \mu}}  \prod_{i,j, i< j} p(h_i,h_j)  dF(h_1) \ldots dF(h_k). \end{equation}
Observe that since the support of $H$ is restricted to $[1,\sqrt{n \mu}]$, the product in \eqref{v1} can be brought into the form
\begin{equation}\label{decomp}
 \prod_{i,j, i< j} p(h_i,h_j)=\prod_{i=1}^k \left(\frac{h_i}{\sqrt{n \mu}}\right)^{k-1},
 \end{equation}
 and hence  \eqref{v1:1} grows as
\begin{align}
A_k(n)&\sim n^k n^{-k(k-1)/2}\left(\int_{1}^{\sqrt{n \mu}} h_1^{k-1} dF(h_1)\right)^k\label{firsth}\nonumber\\
& \sim n^k n^{-k(k-1)/2} \mu^{-k(k-1)/2} \left(\frac{\tau-1}{k-\tau}(\sqrt{n \mu})^{k-1}\oF(\sqrt{n \mu})\right)^k \nonumber\\
& \sim \left(\frac{\tau-1}{k-\tau}\right)^k \mu^{\frac{k}{2}(1-\tau)} n^{\frac{k}{2}(3-\tau)}l^k(\sqrt{n}) \quad {\rm as} \quad n\to\infty,
\end{align}
where by $g_1(n) \sim g_2(n)$ here and throughout we will understand $g_1(n)/g_2(n) \to 1$ as $n \to \infty$. To evaluate the integral in \eqref{firsth} we have invoked Lemma~\ref{lemma:BGT1} in Section~\ref{sec:leading}, which is a simple corollary of \cite[Proposition 1.5.8]{Bingham1989}. We have also used the defining property of a slowly varying function, namely $l(ch) \sim l(h)$ as $h \to \infty$ for any fixed $0 < c < \infty$.

In the remainder of this paper we will not impose a cutoff, so that the random graph has degrees with a truly heavy-tailed distribution. The straightforward reasoning above is then obstructed by the $\min$-operator in the connection probability  \eqref{eq:phh}. Indeed, the product of all connection probabilities can no longer be decomposed as in \eqref{decomp}. 

\vspace{.2cm}

\noindent{\bf Main contributions}.
To deal with the product of all connection probabilities we introduce a specific way of conditioning on the vertex weights. For each of the $k$ vertices that participates in the $k$-clique we condition on whether its weight is smaller or larger than $\sqrt{n \mu}$. This in total yields $k+1$ different configurations: all weights smaller than $\sqrt{n \mu}$, all weights larger than $\sqrt{n \mu}$, and one up to $k-1$ weights smaller than $\sqrt{n \mu}$. The first two configurations (referred to as `extreme configurations') are relatively easy to deal with: all weights smaller than $\sqrt{n \mu}$ corresponds to the cutoff setting, and all weights larger than $\sqrt{n \mu}$ completely eliminates the product of connection probabilities (all equal to one). The remaining configurations (referred to as `middle configurations') are harder to deal with. Based on this idea of conditioning, we establish the following results:
\begin{enumerate}
\item 
In Section~\ref{sec:leading} we obtain the asymptotic behavior of the extreme configurations, and show that the contributions of the middle configurations are asymptotically bounded by the contributions of the extreme configurations. In this way, we circumvent analyzing explicitly the middle configurations, and we obtain the leading-order asymptotics for the average number of cliques in Theorem~\ref{thm:main_clique}.

\item
We then turn in Section~\ref{sec:leading}  to the average number of cycles. The required conditioning is not limited to the vertex weights being smaller or larger than $\sqrt{n \mu}$, but also takes into account how these vertices are arranged on the cycle, making the analysis considerably more difficult. In fact, we first provide in Section~\ref{sec:leading} a lower bound in Theorem~\ref{thm:cycle_lower} on the cycle count for even values of $k$ by considering one specific arrangement (both in terms of size and order) of vertices on a cycle. This lower bound is of interest as it shows that the number of even-sized cycles strictly dominates that of same-sized cliques.

\item
We present in Section~\ref{sec:refined} a more detailed asymptotic analysis of the middle configurations, which in turn leads to sharp asymptotics for the average number of cliques
(Theorems~\ref{thm5.2} and \ref{thm5.5}). For analytic tractability we restrict to the pure power-law case  $F'(h)=c h^{-\tau}$, $h\geq 1$ (with the slowly-varying function taken as a constant).

\item For cycles, the asymptotic evaluation of the integrals involves the theory of circulant matrices (Theorem~\ref{thm5.6}). It turns out that the relevant circulant matrix is regular for odd $k$, and singular for even $k$, leading to different asymptotics for the average number of cliques in Theorems~\ref{thm5.7} and \ref{thm5.8}. The number of even-sized cycles 
are shown to grow faster than even-sized cliques, while odd-sized cycles and odd-sized cliques have comparable growth rates. 

\end{enumerate}
\vspace{.2cm}

\noindent{\bf Relation with other work}.
Our results complement  two existing lines of work. For the degree distribution $
\Prob{H>h}=c h^{1-\tau}$ with a support truncated at the cutoff $\sqrt{n \mu} $,
Bianconi and Marsili obtained sharp asymptotics for both clique counts \cite{bianconi2006emergence} and cycle counts \cite{bianconi2005loops}. The main extension in this paper is that we do not impose the cutoff $\sqrt{n \mu}$, as explained above, and hence work with truly heavy-tailed weight distributions.

The other line of work was launched recently by Van der Hofstad {\it et al.}~\cite{hofstad2017d}, who consider $\Prob{H>h}=c h^{1-\tau}$ with infinite support, and study the optimal composition for the most likely subgraph. They showed that for a large class of subgraphs, including cliques and cycles, the optimal composition consists exclusively of vertices with order $\sqrt{n}$ degrees. They also showed that this optimal composition determines up to leading order the asymptotic growth of the average number of subgraphs as function of the network size $n$. We sharpen the asymptotics obtained  in \cite{hofstad2017d} by directly analyzing the integral expressions for the average number of cycles and cliques. We restrict to cliques and cycles, and do not consider all possible subgraphs as in \cite{hofstad2017c}, because we utilize the specific topological structure of cliques and cycles in ways that cannot be easily generalized. 

Apart from sharpening results in \cite{bianconi2006emergence,bianconi2005loops,hofstad2017d}, we sometimes consider a more general setting with the slowly-varying function $l(h)$ in \eqref{eq:pl}, which allows for deviations from the pure power law \cite{welldone}. 
The general consensus is that the exact shape of $l(h)$ is less important than the precise value of $\tau$.
For $A_k(n)$ and $C_k(n)$, $\tau$ indeed determines the leading growth rate, but $l(h)$ enters the asymptotic expressions in various non-trivial ways.


\section{Rough asymptotics}\label{sec:leading}

By $g_1(n) \asymp g_2(n)$ as $n \to \infty$ we are going to understand that there exist constants $C_1 > 0$ and $C_2 < \infty$ such that
\be
C_1 \le \lim\inf_{n\to \infty} \frac{g_1(n)}{g_2(n)} \le \lim\sup_{n\to \infty} \frac{g_1(n)}{g_2(n)} \le C_2.
\ee
We write $g_1(n) \lesssim g_2(n)$ if there exists a constant $C < \infty$ such that
\be
\lim\sup_{n\to \infty} \frac{g_1(n)}{g_2(n)} \le C.
\ee
We also write $g_1(n) \gtrsim g_2(n)$ if there exists a constant $C > 0$ such that
\be
\lim\inf_{n\to \infty} \frac{g_1(n)}{g_2(n)} \ge C.
\ee
Recall that we write $g_1(n)\sim g_2(n)$ when $g_1(n)/g_2(n)\to 1$ as $n\to\infty.$

We write the tail of the degree distribution as $\oF(h) = h^{-\tau+1} l(h)$ with $\tau \in (2,3)$ and $l(h)$  a slowly-varying function.
Note that 
\begin{align}
\oF(\sqrt{n \mu}) &= (\sqrt{\mu})^{1-\tau} (\sqrt{n})^{1-\tau} l(\sqrt{n \mu}) \sim (\sqrt{\mu})^{1-\tau} (\sqrt{n})^{1-\tau} l(\sqrt{n})\nonumber\\ 
&\asymp \oF(\sqrt{n}) \quad {\rm as} \quad  n \to \infty.
\end{align}

\subsection{Rough asymptotics for cliques}

\begin{thm}[Rough asymptotics for cliques] \label{thm:main_clique}
In the rank-1 inhomogeneous random graph with weight distribution \eqref{eq:pl}
and connection probability \eqref{eq:phh}, the average number of cliques of size $k \ge 3$ scales as 
\begin{equation}\label{ras}
A_k(n) \asymp n^k \left(\Fb(\sqrt{n})\right)^{k} = n^{\frac{k}{2}(3-\tau)}l^k(\sqrt{n}) \quad {\rm as} \quad  n \to \infty.
\end{equation}
\end{thm}

Comparing the rough asymptotics \eqref{ras} with 
\eqref{firsth}, we see that imposing a cutoff does not change the leading growth rate $n^{\frac{k}{2}(3-\tau)} l^k(\sqrt{n})$ and it is only the constant that may change. We defer calculating the exact constant to Section~\ref{sec:refined}. Because we already know that the configuration with all weights smaller than $\sqrt{n \mu}$ gives the asymptotics in \eqref{ras}, the goal of the present section is to demonstrate that the contributions of all other configurations of vertex weights (such that at least one of them is larger than $\sqrt{n \mu}$) does not exceed $n^{\frac{k}{2}(3-\tau)} l^k(\sqrt{n})$ asymptotically.

Before we present the proof of Theorem~\ref{thm:main_clique}, observe that the probability of having an edge between vertices $i$ and $j$ can be described as
\be
\P\left(\frac{H_i H_j}{n\mu} > U_{ij}\right), 
\ee
where $H_1,\ldots, H_n$ are independent copies of $H$ and $U_{ij}$ are independent $U(0,1)$ random variables. The model may therefore be thought of as follows: Given a collection of random variables $H_i$ and $U_{ij}$, an edge between vertices $i$ and $j$ is present if
$
{H_i H_j}/{n\mu} > U_{ij}
$.
This may be rewritten as
\be
H_i H_j V_{ij} > n\mu,
\ee
where $V_{ij} = 1/U_{ij}$ has a Pareto$(1)$ distribution: $\P(V_{ij} > x) = 1/x$ for all $x \ge 1$. 
The average number of edges (cliques of size $2$) is then straightforward:
\be
{n\choose 2} \P(H_1 H_2 V > n \mu) \sim n^2 \frac{1}{n \mu} = \frac{n}{\mu},
\ee
as we have a product of three independent regularly varying random variables such that two of them ($H_1$ and $H_2$) are much lighter than the third ($V$), and the result follows since $H$ is regularly varying with a finite mean.
For a general $k$, we see that
\begin{equation}\label{pk}
\P(K_k) =\P(H_i H_j V_{ij} > n \mu, 1 \le i \le k, 1 \le j \le k, i \neq j).
\end{equation}
The proof strategy for Theorem~\ref{thm:main_clique} is to obtain large-$n$ asymptotics for 
$\P(K_k)$ by using the conditioning explained in Section~\ref{intro} and properties of random variables with regularly varying distributions, including the following lemma.

\begin{lem} \label{lemma:BGT1}
{\rm (i)} If $\beta > \tau-1$, then
\be
\int_{1}^x h^{\beta} dFh) \sim \frac{\tau-1}{\beta-\tau+1} x^{\beta} \oF(x) \quad {\rm as} \quad  x \to \infty.
\ee
{\rm (ii)} If $\beta < \tau-1$, then
\be
\int_{x}^\infty h^{\beta} dF(h) \sim \frac{\tau-1}{\tau-\beta-1} x^{\beta} \oF(x) \quad {\rm as} \quad  x \to \infty.
\ee
\end{lem}
\begin{proof}
(i) Integration by parts leads to
\be \label{lemma:BGT2}
\int_{1}^x h^{\beta} dF(h) = 1 - x^\beta \oF(x) + \beta\int_1^x h^{\beta-\tau} l(h)dh,
\ee
and the integral on the right-hand side  is asymptotically equivalent to
\be
\beta\int_1^x h^{\beta-\tau} l(h)dh \sim \frac{\beta}{\beta-\tau+1} x^{\beta-\tau+1} l(x) = \frac{\beta}{\beta-\tau+1} x^\beta \oF(x)
\ee
thanks to \cite[Proposition 1.5.8]{Bingham1989}. 
The proof of (ii) follows the same lines as (i) and uses \cite[Proposition 1.5.10]{Bingham1989}.
\end{proof}

\noindent{\bf Proof of Theorem \ref{thm:main_clique}.} 
Denote $\gamma_n = \sqrt{n \mu}$ and for $1 \le l \le k$
\begin{equation}
A_l = \{H_1 \le \gamma_n,\ldots, H_l \le \gamma_n\} \quad {\rm and} \quad 
B_l = \{H_l > \gamma_n, \ldots, H_k > \gamma_n\}.
\end{equation}
We focus on the probability
\begin{align}
& \P(H_i H_j V_{ij} > \gamma_n^2, 1 \le i,j \le k, i \neq j) \nonumber
\\ & = \P(A_k, \{H_i H_j V_{ij} > \gamma_n^2, 1 \le i,j \le k, i \neq j\}) + \sum_{m=1}^{k-1} {k \choose m} I_m
 + \P(B_1), \label{eq:S1_regular}
\end{align}
where
\be
I_m = \P(A_m, B_{m+1}, \{H_i H_j V_{ij} > \gamma_n^2, 1 \le i,j \le k, i \neq j\}),
\ee
and in order to prove the theorem, it is sufficient to show that the probability of the left-hand side  of \eqref{eq:S1_regular} behaves asymptotically, in terms of the leading term, as $\left(\oF(\sqrt{n}\right)^k$. We refer to the first and last terms in \eqref{eq:S1_regular} as `extreme configurations' and the summands with $m=1,\dots,k-1$ as `middle configurations'. 

We have seen that the first summand (extreme configuration) behaves asymptotically as
\begin{align}
\left(\frac{\tau-1}{k-\tau}\right)^k \mu^{\frac{k}{2}(1-\tau)} \left(\oF(\sqrt{n})\right)^k \asymp \left(\oF(\sqrt{n})\right)^k. 
\end{align}
The last summand (contribution of the second extreme configuration) is clearly equal to
$
\left(\oF(\gamma_n)\right)^k \asymp \left(\oF(\sqrt{n})\right)^k
$.
In order to prove Theorem \ref{thm:main_clique}, we then need to show that the contributions of all middle configurations in \eqref{eq:S1_regular} do not exceed that of the extreme configurations,
and are hence bounded from above by $C \left(\oF(\sqrt{n})\right)^k$ with some constant $C < \infty$.

For $m\ge 3$ we have
\begin{align}
& \P(A_m, B_{m+1}, \{H_i H_j V_{ij} > \gamma_n^2, 1 \le i,j \le k, i \neq j\})
\nonumber\\ & \le \P(A_m, \{H_i H_j V_{ij} > \gamma_n^2, 1 \le i,j \le m, i \neq j\}) \P(B_{m+1})
\nonumber\\ & \asymp \left(\oF(\sqrt{n})\right)^{m} \left(\oF(\sqrt{n})\right)^{k-m}= \left(\oF(\sqrt{n})\right)^{k},
\end{align}
due to what we have already shown.

For the cases $m=1$ and $m=2$, we are going to use \cite[Theorem 1.5.4]{Bingham1989}, which implies that there exists a non-increasing function $\psi(x)$ such that $x \oF(x) \sim \psi(x)$ as $x \to \infty$. For $m=1$, note that it is sufficient to consider the case $k=3$ and show that
\be \label{eq:suff_3_regular_1}
\P(H_1 \le \gamma_n, H_2 > \gamma_n, H_3 > \gamma_n, H_i H_j V_{ij} > \gamma_n^2, i \neq j) \lesssim \left(\oF(\sqrt{n})\right)^3.
\ee
Indeed, if \eqref{eq:suff_3_regular_1} holds, then for a general $k$ we have the following estimate:
\begin{align}
& \P(\{H_1 \le \gamma_n\}, B_2, \{H_i H_j V_{ij} > n, 1 \le i,j \le k, i \neq j\})
\nonumber\\ & \le  \P(H_1 \le \gamma_n, H_2 > \gamma_n, H_3 > \gamma_n, H_i H_j V_{ij} > \gamma_n^2, i \neq j) \P(\{H_i > \gamma_n, 4 \le i \le k\})
\nonumber \\ & \lesssim (\oF(\sqrt{n}))^3 (\oF(\sqrt{n}))^{k-3} \lesssim (\oF(\sqrt{n}))^k.
\end{align}
Therefore, for the case $m=1$, it remains to prove \eqref{eq:suff_3_regular_1}.
Write
\begin{align}
&\P(H_1 \le \gamma_n, H_2 > \gamma_n, H_3 > \gamma_n, H_i H_j V_{ij} > \gamma_n^2, i \neq j)
\nonumber \\ & =  \int_1^{\gamma_n} dF(h_1) \left(\int_{\gamma_n}^{\gamma_n^2/h_1} \frac{h_1 h_2}{\gamma_n^2} dF(h_2) + \oF\left(\frac{\gamma_n^2}{h_1}\right) \right)^2 \nonumber\\
&\lesssim \int_1^{\gamma_n} dF(h_1)\left(\frac{h_1}{\gamma_n} \oF\left(\gamma_n\right) + \oF\left(\frac{\gamma_n^2}{h_1}\right) \right)^2.
\end{align}
Note now that
\be
\frac{\gamma_n^2}{h_1} \oF\left(\frac{\gamma_n^2}{h_1}\right) \sim \psi\left(\frac{\gamma_n^2}{h_1}\right) \le \psi(\gamma_n) \sim \gamma_n \oF(\gamma_n),
\ee
as $h_1 \le \gamma_n$, and hence
\begin{align}
& \P(H_1 \le \gamma_n, H_2 > \gamma_n, H_3 > \gamma_n, H_i H_j V_{ij} > \gamma_n^2, i \neq j)
\nonumber\\ & \lesssim  \frac{1}{\gamma_n^2}\left(\oF\left(\gamma_n\right)\right)^2\int_1^{\gamma_n} h_1^2 dF(h_1) \lesssim \frac{1}{\gamma_n^2}\left(\oF\left(\sqrt{n}\right)\right)^2 \gamma_n^2 \oF\left(\gamma_n\right) = \left(\oF\left(\sqrt{n}\right)\right)^3,
\end{align}
where we used Lemma \ref{lemma:BGT1}. 
It now remains to consider the case $m=2$. Note that, again, it is sufficient to consider the case $k=3$ and show that
\be \label{eq:suff_3_regular}
\P(H_1 \le \gamma_n, H_2 \le \gamma_n, H_3 > \gamma_n, H_i H_j V_{ij} > \gamma_n^2, i \neq j) \lesssim (\oF(\sqrt{n}))^3.
\ee
Indeed, if this is the case, we can write
\begin{align}
& \P(A_2, B_3, \{H_i H_j V_{ij} > \gamma_n^2, 1 \le i,j \le k, i \neq j\})
\nonumber\\ & \le \P(H_1 \le \gamma_n, H_2 \le \gamma_n, H_3 > \gamma_n, \{H_i H_j V_{ij} > \gamma_n^2, i\neq j\}) \P(\{H_i > \gamma_n, 4 \le i \le k\})
\nonumber\\  & \lesssim (\oF(\sqrt{n}))^3 (\oF(\sqrt{n}))^{k-3} \lesssim  (\oF(\sqrt{n}))^k,
\end{align}
which is sufficient. Therefore it remains to prove \eqref{eq:suff_3_regular}. Write
\begin{align} \label{eq:case_m2}
& \P(H_1 \le \gamma_n, H_2 \le \gamma_n, H_3 > \gamma_n, H_i H_j V_{ij} > \gamma_n^2, i \neq j) 
\nn \\ & = 2 \P(H_1 \le \gamma_n, H_2 \le H_1, H_3 > \gamma_n, H_i H_j V_{ij} > \gamma_n^2, i \neq j)
\nn 
\\ & = \int_1^{\gamma_n} dF(h_1) \int_1^{h_1} dF(h_2) \frac{h_1 h_2}{\gamma_n^2} \left(R_1(h_1, h_2) + R_2(h_1, h_2) + R_3(h_1, h_2) \right),
\end{align}
where
\be
R_1(h_1, h_2) = \int_{\gamma_n}^{\gamma_n^2/h_1} \frac{h_1 h_2 h_3^2}{\gamma_n^4} dF(h_3) \asymp \frac{h_1 h_2}{\gamma_n^4}\left(\frac{\gamma_n^2}{h_1}\right)^{2}\oF\left(\frac{\gamma_n^2}{h_1}\right) = \frac{h_2}{h_1}\oF\left(\frac{\gamma_n^2}{h_1}\right),
\ee
\be
R_2(h_1, h_2) = \int_{\gamma_n^2/h_1}^{\gamma_n^2/h_2} \frac{h_2 h_3}{\gamma_n^2} dF(h_3) = \frac{h_2}{\gamma_n^2} \frac{\gamma_n^2}{h_1} \oF\left(\frac{\gamma_n^2}{h_1}\right) = R_1(h_1,h_2),
\ee
and
\be
R_3(h_1, h_2) = \oF\left(\frac{\gamma_n^2}{h_2}\right).
\ee
Therefore, in order to bound the right-hand side of \eqref{eq:case_m2} from above, we need to bound the integrals involving $R_1$ and $R_3$. The integral involving $R_1$ may be rewritten as
\begin{align}
\frac{1}{\gamma_n^2} \int_1^{\gamma_n} \oF(\gamma_n^2/h_1) dF(h_1)\int_1^{h_1} h_2^2 dF(h_2) & \lesssim \frac{1}{\gamma_n^2}\int_1^{\gamma_n}h_1^{2} \oF(\gamma_n^2/h_1) \oF(h_1) dF(h_1)\nn \\ 
& \asymp n^{-\tau}\int_1^{\gamma_n} h_1^2 l(\gamma_n^2/h_1)l(h_1) dF(h_1). 
\end{align}
Take $\delta>0$ and note that (due to \cite[Theorem 1.5.6]{Bingham1989}) there exists a constant $A$ such that 
\be \label{eq:bound_svf_1}
\frac{l(\gamma_n^2/h_1)}{l(\gamma_n)} \le A \left(\frac{\gamma_n}{h_1} \right)^{\delta}
\ee 
and
\be \label{eq:bound_svf_2}
\frac{l(h_1)}{l(\gamma_n)} \le A \left(\frac{h_1}{\gamma_n}\right)^\delta
\ee
for large enough $n$ and for all $h_1 \le \gamma_n$. Using \eqref{eq:bound_svf_1} and \eqref{eq:bound_svf_2}, we can then bound the last integral with the following expression:
\be
A^2 n^{-\tau} \left(l(\gamma_n)\right)^2 \int_1^{\sqrt{n}}h_1^2 dF(h_1) \lesssim n^{-\tau} \left(l(\sqrt{n})\right)^2 (\sqrt{n})^2 \oF(\sqrt{n}) \lesssim \left(\oF(\sqrt{n})\right)^3.
\ee
It now remains to bound the integral on the right-hand side of \eqref{eq:case_m2} involving $R_3$:
\begin{align}
& \int_1^{\gamma_n} dF(h_1) \int_1^{h_1} dF(h_2) \frac{h_1 h_2}{\gamma_n^2} \oF\left(\frac{\gamma_n^2}{h_2}\right)
\nn \\ & = \frac{1}{\gamma_n^2} \int_1^{\gamma_n} h_2 \oF\left(\frac{\gamma_n^2}{h_2}\right) dF(h_2) \int_{h_2}^{\gamma_n} h_1 dF(h_1) 
\nn \\ & \asymp \frac{1}{\gamma_n^2} \int_1^{\gamma_n} h_2^2 \oF(h_2) \oF\left(\frac{\gamma_n^2}{h_2}\right) dF(h_2),
\end{align}
which is asymptotically equivalent to the integral involving $R_1$. \qed

\subsection{Rough asymptotics for cycles}

The following theorem illustrates that cycles with an even number of vertices are more likely than cliques with the same number of nodes. The proof is constructive as it presents a particular configuration resulting in the leading asymptotics. In Section~\ref{sec:refined} we present precise asymptotic analysis using a different, more involved  method.

As was pointed out in the introduction, the asymptotic analysis of the number of cycles is more difficult than that of the number of cliques, as not only the weights of vertices matter but also their locations in the cycle. To simplify the analysis, we therefore restrict attention to the case of weight distributions having regularly varying densities.

\begin{thm}[Lower bound cycle asymptotics] \label{thm:cycle_lower}
In the rank-1 inhomogeneous random graph with weight density $\rho(h) = h^{-\tau} l(h)$, average weight $\mu$
and connection probability \eqref{eq:phh}, the average number of cycles of even  size $k \ge 4$ satisfies
\begin{equation} \label{eq:cycle_even}
C_k(n) \gtrsim n^k  \left(\Fb(\sqrt{n})\right)^{k} \int_{\sqrt{n \mu}}^{n \mu} \frac{1}{h} \frac{l^{k/2}(h) l^{k/2}(n \mu/h)}{l^k(\sqrt{n \mu})} dh.
\end{equation}
\end{thm}

\begin{remark}\label{remm1}
Theorem \ref{thm:cycle_lower} implies that the asymptotics of $C_k(n)$ are heavier than the asymptotics of the number of cliques $A_k(n)$.
Indeed, fix $a > \sqrt{\mu}$. Then for $n\mu\geq a\sqrt{n}$, i.e., $n\geq (a/\mu)^2$,
\begin{align}
&\int_{\sqrt{\mu n}}^{n \mu} \frac{1}{h} \frac{l^{k/2}(h) l^{k/2}(n \mu/h)}{l^k(\sqrt{n \mu})} dh  \ge  \int_{\sqrt{\mu n}}^{a \sqrt{n}} \frac{1}{h} \frac{l^{k/2}(h) l^{k/2}(n \mu/h)}{l^k(\sqrt{n \mu})} dh
\nn\\ & \sim  \int_{\sqrt{\mu n}}^{a \sqrt{n}} \frac{1}{h}\frac{l^{k/2}(\sqrt{n}) l^{k/2}(\sqrt{n})}{l^k(\sqrt{n})} dh =\int_{\sqrt{ \mu n}}^{a \sqrt{n}} \frac{1}{h} dh=  \log a - \tfrac12 \log \mu,
\end{align}
where in the second step we have used that $n\mu /h\in [\mu\sqrt{n}/a,\sqrt{\mu n}]$ when $h\in[\sqrt{\mu n},a\sqrt{n}]$ so that both $l(h)/l(\sqrt{n})\to 1$ and 
$l(n \mu/h)/l(\sqrt{n})\to 1$ uniformly in $h\in[\sqrt{\mu n},a\sqrt{n}]$ as $n\to\infty$. 
As $a$ is arbitrary, $\log a - \tfrac12 \log \mu$ may be  arbitrarily large.
\end{remark}

\noindent{\bf Proof of Theorem \ref{thm:cycle_lower}.} 
From \eqref{defcycle} we see that 
$
C_k(n)  \asymp n^k \P(C_k)$, so it is 
sufficient to show that
\be
\P(C_k) \gtrsim \left(\Fb(\sqrt{n})\right)^{k} \int_{\sqrt{n \mu}}^{n \mu} \frac{1}{h} \frac{l^{k/2}(h) l^{k/2}(n \mu/h)}{l^k(\sqrt{n \mu})} dh.
\ee

We again use notation $\gamma_n = \sqrt{n \mu}$. Fix a $k=2m$ for an integer $m$ and write (with the convention that by index $k+1$ we understand index $1$). Denote by $A = \{\gamma_n <H_{2i-1} < \gamma_n^2, 1 \le i \le m\}$ and by $B = \{\{H_{2j} \le \gamma_n, 1 \le j \le m\}$ and write
\begin{align}
& \P(C_k) = \P(H_i H_{i+1} V_{i,i+1} > \gamma_n^2, 1 \le i \le k) 
\nn\\ & \ge \P(A, B, \{H_i H_{i+1} V_{i,i+1} > \gamma_n^2, 1 \le i \le k\})
\nn\\ & \ge \prod_{i=1}^{m} \int_{\gamma_n}^{\gamma_n^2} \rho(h_{2i-1}) dh_{2i-1} \prod_{j=1}^m \int_1^{\gamma_n^2/M_j} \frac{h_{2j}^2 h_{2j-1} h_{2j+1}}{\gamma_n^4} \rho(h_{2j}) dh_{2j},
\end{align}
with $M_j = \max\{h_{2j-1},h_{2j+1}\}$ (we will also use $m_j = \min\{h_{2j-1},h_{2j+1}\}$). Thanks to Lemma \ref{lemma:BGT1}, the previous terms may be bounded from below by
\begin{align}
& \prod_{i=1}^{m} \int_{\gamma_n}^{\gamma_n^2} \rho(h_{2i-1}) dh_{2i-1} \prod_{j=1}^m \frac{h_{2j-1} h_{2j+1}}{\gamma_n^4} \left(\frac{\gamma_n^2}{M_j} \right)^{-\tau+3} l\left(\frac{\gamma_n^2}{M_j}\right)
\nn\\ & \asymp n^{k/2(-\tau+1)} \prod_{i=1}^{m} \int_{\gamma_n}^{\gamma_n^2} \rho(h_{2i-1}) dh_{2i-1} \prod_{j=1}^m \frac{m_j}{M_j} \left(M_j\right)^{\tau-1} l\left(\frac{\gamma_n^2}{M_j}\right)
\nn\\ & \ge n^{k/2(-\tau+1)} \int_{\gamma_n}^{\gamma_n^2} \rho(h_1) dh_1 \int_{\gamma_n}^{h_1} \rho(h_3) dh_3 \ldots \int_{\gamma_n}^{h_{k-3}} \rho(h_{2k-1}) dh_{2k-1} 
\nn\\ & \left(h_{2k-1}^2 h_1^{-2} h_1^{2(\tau-1)} h_3^{\tau-1} \ldots h_{2k-3}^{\tau-1} l^2\left(\frac{\gamma_n^2}{h_1}\right) l\left(\frac{\gamma_n^2}{h_3}\right) \ldots l\left(\frac{\gamma_n^2}{h_{2k-3}}\right)\right).
\end{align}
Note now that the above integral with respect to $h_{2k-1}$  reads
\begin{align}
\int_{\gamma_n}^{h_{2k-3}} \rho(h_{2k-1}) h_{2k-1}^2 dh_{2k-1} & = \int_{\gamma_n}^{h_{2k-3}} h_{2k-1}^{-\tau+2} l(h_{2k-1}) dh_{2k-1} \nn \\ & \gtrsim h_{2k-3}^{-\tau+3} l(h_{2k-3}),
\end{align}
and hence the integral with respect to $h_{2k-3}$ satisfies
\begin{align}
& \int_{\gamma_n}^{h_{2k-5}} \rho(h_{2k-3}) h_{2k-3}^2 l(h_{2k-3}) l\left(\frac{\gamma_n^2}{h_{2k-3}}\right) dh_{2k-3} 
\nn\\ & = \int_{\gamma_n}^{h_{2k-5}} h_{2k-3}^{-\tau+2} l^2(h_{2k-3}) l\left(\frac{\gamma_n^2}{h_{2k-3}}\right) dh_{2k-3}
\nn\\ & \gtrsim l\left(\frac{\gamma_n^2}{h_{2k-5}}\right) \int_{\gamma_n}^{h_{2k-5}} h_{2k-3}^{-\tau+2} l^2(h_{2k-3}) \left(\frac{h_{2k-5}}{h_{2k-3}}\right)^{-\delta} dh_{2k-3}
\nn\\ & \gtrsim l\left(\frac{\gamma_n^2}{h_{2k-5}}\right) h_{2k-5}^{-\delta} h_{2k-5}^{-\tau+3+\delta} l^2(h_{2k-5}) = h_{2k-5}^{-\tau+3} l^2(h_{2k-5})l\left(\frac{\gamma_n^2}{h_{2k-5}}\right),
\end{align}
where we used \cite[Theorem 1.5.6]{Bingham1989}. It is easy to see how this may be continued by induction to show that
\begin{align}
\P(C_k) & \gtrsim  n^{k/2(-\tau+1)} \int_{\gamma_n}^{\gamma_n^2} \rho(h_1) dh_1 h_1^{-2} h_1^{2(\tau-1)} h_1^{-\tau+3} l^{k/2-1}(h_1)l^{k/2}\left(\frac{\gamma_n^2}{h_1}\right) 
\nn\\ & =  n^{k/2(-\tau+1)} \int_{\gamma_n}^{\gamma_n^2} h_1^{-1} l^{k/2}(h_1)l^{k/2}\left(\frac{\gamma_n^2}{h_1}\right) dh_1,
\end{align}
which completes the proof. \qed
\section{Precise asymptotics}\label{sec:refined}

We now consider the pure power-law case
\beq \label{5.2}
\rho(h)=F'(h)=c\,h^{-\tau},\quad h\geq 1,
\eq
with $c=\tau-1$, and derive more precise asymptotic results for the 
average number of $k$-cliques $A_k(n)$ and $k$-cycles $C_k(n)$ in the large-network limit $n\to\infty$. We choose to work with the pure power law density, instead of the regularly varying distribution function \eqref{eq:pl}, to suppress notation in view of the elaborate calculations that will follow. 
We again use as short-hand notation
$
\cutoff=\sqrt{n\mu}
$ with $\mu=(\tau-1)/(\tau-2)$, 
so that the connection probability  in \eqref{eq:phh} can be writen as
\begin{equation}\label{eq:phhgeneral}
		p(h_i,h_j)=f\left(\frac{h_i h_j}{\cutoff^2}\right),
\end{equation}
with $f(x)=\min\,\{1,x\}$. More generally, in Subsection \ref{subsec5.2} we present results that hold for the class of continuous nonnegative nondecreasing functions $f$ considered in \cite{hofstad2017b} that satisfy 
\beq \label{5.5}
\frac{f(x)}{x}\pr 1\,,~~x\downarrow 0~;~~~~~~f(x)\pr 1\,,~~x\pr\infty. 
\eq
Observe that for $x\geq 0$, the function $f(x)=\min\,\{1,x\}$ belongs to this class, and so do other standard choices like $f(x)={x}/{(1+x)}$ and $f(x)=1-{\rm e}^{-x}$.

As before, the notation $g_1(n)\sim g_2(n)$ is used for $g_1(n)=g_2(n)(1+o(1))$ as $n\pr\infty$.


\subsection{Precise asymptotics for cliques} \label{subsec5.1}

With $k\geq 3$, the average number $A_k(n)$ of $k$-cliques equals
\beq \label{5.7}
A_k(n)=\Bigl(\!\ba{c} n \\ k \ea\!\Bigr)\,\il_1^{\infty}\cdots \il_1^{\infty}\,\rho(h_1)\cdots\rho(h_k)\,\prod_{1\leq i<j\leq k}\,f\Bigl(\frac{h_ih_j}{\cutoff^2}\Bigr)\,dh_1\cdots dh_k. 
\eq
To analyze this $k$-fold integral we make the choice $f(x)=\min\{1,x\}$, $x\geq0$. We split the integral in (\ref{5.7}) into $k+1$ integrals over subranges where precisely $m$ of the hidden variables $h_i$ are $\leq\, \cutoff$ while the $k-m$ others are $\geq \,\cutoff$, $m=0,1,...,k$. Observe that this range splitting is exactly the same as the conditioning used in \eqref{eq:S1_regular}. By symmetry of the integrand we have
\beq \label{5.8}
A_k(n)=\Bigl(\!\ba{c} n \\ k \ea\!\Bigr)\,\sum_{m=0}^k\,\Bigl(\!\ba{c} k \\ m \ea\!\Bigr)\,I_m, 
\eq
where $I_m$ is the contribution of the range
\beq \label{5.9}
1\leq h_1,h_2,...,h_m\leq \cutoff\leq h_{m+1},h_{m+2},...,h_k\leq h_c. 
\eq
By the choice $f(x)=\min\{1,x\}$, we have
\beq \label{5.10}
I_0\sim \cutoff^{k(1-\tau)}, ~~~~~~I_k\sim \cutoff^{k(1-\tau)}\Bigl(\frac{\tau-1}{k-\tau}\Bigr)^k. 
\eq
For $m=1,2,...,k-1$ we have
\begin{align} \label{5.11}
& \mbox{}  I_m=
\begin{array}[t]{c} \underbrace{\il_1^{\cutoff}\cdots\il_1^{\cutoff}} \\ m \end{array}
\begin{array}[t]{c} \underbrace{\il_{\cutoff}^{\infty}\cdots\il_{\cutoff}^{\infty}} \\ k-m \end{array}
\,\rho(h_1)\cdots\rho(h_m)\,\rho(h_{m+1}\cdots\rho(h_k) \nonumber \\ 
& \hspace*{2.8cm}\times\prod_{1\leq i<j\leq k}\,f\Bigl(\frac{h_ih_j}{\cutoff^2}\Bigr)\,dh_1\cdots dh_mdh_{m+1}\cdots dh_k. 
\end{align}
The main result for $I_m$ reads as follows.

\begin{prop} \label{prop5.1}
For $m=1,2,...,k-1$
\beq \label{5.12}
I_m \sim \cutoff^{k(1-\tau)}(\tau-1)^m\,m!\,J_m,
\eq 
where
\beq \label{5.13a}
J_m=\il_0^1\cdots\il_0^1\,\Bigl(\prod_{i=1}^m\,t_i^{i(m-\tau)-1}\Bigr)\,\Phi_m^{k-m}(t_1,...,t_m)\,dt_1 \cdots dt_m, 
\eq
with
\beq \label{5.13}
\Phi_1(t_1)=t_1\Bigl(\frac{\tau-1}{\tau-2}-\frac{t_1^{\tau-2}}{\tau-2}\Bigr)\:,~~~~0\leq t_1\leq1, 
\eq
\beq \label{5.14}
\Phi_2(t_1,t_2)=t_1t_2^{\tau-1}\Bigl(\frac{\tau-1}{(3-\tau)(\tau-2)}-\frac{t_1^{\tau-2}}{\tau-2}- \frac{\tau-1}{3-\tau}\,t_2^{3-\tau}\Bigr)\:,~~~~0\leq t_1,t_2\leq 1, 
\eq
and, for $m=3,...,k-1\,$ and $0\leq  t_1,...,t_m\leq1$,
\begin{align} \label{5.15}
\Phi_m(t_1,...,t_m)  &=  t_1t_2^{\tau-1}\cdots t_m^{\tau-1}\,\Bigl(\frac{\tau-1}{(3-\tau)(\tau-2)} -\frac{t_1^{\tau-2}}{\tau-2} \nonumber \\
&  -~\sum_{j=3}^m\,\frac{\tau-1}{(j+1-\tau)(j-\tau)}\,t_2^{3-\tau}t_3^{4-\tau}\cdots t_{j-1}^{j-\tau} \nonumber \\&  -~\frac{\tau-1}{m+1-\tau}\,t_2^{3-\tau}t_3^{4-\tau}\cdots t_m^{m+1-\tau}\Bigr).
\end{align}
\end{prop}
\mbox{}

The detailed proof of Proposition \ref{prop5.1} is deferred to Section~\ref{appA}. It uses the basic substitution $v_i=h_i/\cutoff$ in (\ref{5.11}), causing the factor $\cutoff^{k(1-\tau)}$ to emerge, and the special form of $f(x)$ ($=\,\min\{1,x\}$), so that symmetry and factorizations can be exploited. There is, furthermore, an explicit evaluation of the integrals over $v_{m+1},...,v_k$ when $0\leq v_1\leq v_2\leq\cdots\leq v_m\leq1$. A final substitution ($t_i=v_i/v_{i+1}$ for $i=1,...,m-1$ and $t_m=v_m$) then yields integrals over the unit cube $[0,1]^m$.

The form of $I_m$ and $J_m$ in (\ref{5.12}) and (\ref{5.13a}) shows a convenient separation of dependencies, with $J_m$ in \eqref{5.13a} independent of $n$ and $\Phi_m$ in (\ref{5.13})--(\ref{5.15}) independent of $k$. Moreover, the number of integration variables is reduced from $k$ in (\ref{5.11}) to $m$ in (\ref{5.13a}). The remaining integral $J_m$ is not readily computable in closed form. 

From (\ref{5.10}) and Proposition \ref{prop5.1} we get the following result for the average number of $k$-cliques:

\begin{thm}[Precise asymptotics for cliques] \label{thm5.2}
In the rank-1 inhomogeneous random graph with weight density \eqref{5.2}
and connection probability \eqref{eq:phh}, the average number of cliques of size $k \ge 3$ satisfies
\beq \label{5.17}
A_k(n)\sim \Bigl(\!\ba{c} n \\ k \ea\!\Bigr)\,\cutoff^{k(1-\tau)}\,\Bigl[1+\sum_{m=1}^{k-1}\,\Bigl(\!\ba{c} k \\ m \ea\!\Bigr)\,(\tau-1)^m\,m!\,J_m+\Bigl(\frac{\tau-1}{k-\tau}\Bigr)^k\Bigr], 
\eq 
with $J_m$ given in {\rm(\ref{5.13a})}.
\end{thm}

Observing that ${n\choose k}\cutoff^{k(1-\tau)}\sim n^{k(3-\tau)/2}\mu^{k(1-\tau)/2}/k!$, we see that Theorem~\ref{thm5.2} confirms and refines Theorem~\ref{thm:main_clique} for the pure power-law case \eqref{5.2}. Below we give further results on the expression in square brackets in \eqref{5.17}.

The representation (\ref{5.12})--(\ref{5.13a}) of $I_m$ is also useful for getting bounds and asymptotics for $I_m$ and $A_k(n)$. For this there is the following result.

\begin{prop} \label{prop5.3}
For $m=1,2,...,k-1$,
\bi{(iii)~}
\ITEM{\rm (i)} $\Phi_m(0,...,0)=0\,,~~\Phi_m(1,...,1)=1\,$,
\ITEM{\rm (ii)} $\Phi_m(t_1,...,t_m)$ increases in any of the $t_i\in[0,1]\,$,
\ITEM{\rm (iii)} $\dfrac{\partial\Phi_m}{\partial t_i}\,(1,...,1)=0\,,~~i=1,...,m\,$,
\ITEM{\rm (iv)} $\dfrac{\partial^2\Phi_m}{\partial t_i\,\partial t_j}\,(1,...,1)={-}(\tau-1)\,\min\{i,j\}\,,~~i,j=1,...,m\,$.
\ei
\end{prop}

The maximality of $\Phi_m$ at $t_1=\cdots=t_m=1$ translates to $h_1=\cdots=h_m=\cutoff$ in the original hidden variables $h_i$, and this shows that for $m=1,...,k-1$ the largest contribution to the integral $I_m$ in (\ref{5.11}) comes from hidden variables $h_1,...,h_m$ that are less than but near $\cutoff$ while the other hidden variables $h_{m+1},...,h_k$ exceed $\cutoff$.

From Proposition \ref{prop5.3} we have the following consequences for the quantities $(\tau-1)^m\,m!\,J_m$ occurring in the series (\ref{5.17}) for $A_k(n)$.

\begin{prop} \label{prop5.4} \mbox{}\\[-5mm]
\bi{(ii)~}
\ITEM{\rm (i)} $(\tau-1)^m\,m!\,J_m$ decreases in $k=3,4,...$ for $m=1,2,...,k-1$.
\ITEM{\rm (ii)} $(\tau-1)^m\,m!\,J_m\leq\Bigl(\dfrac{\tau-1}{m-\tau}\Bigr)^m$ for $m=3,4,...,k-1$.
\ei
\end{prop}

The series in (\ref{5.17}) for $A_k(n)$ has terms that are bounded by ${k\choose m}(\frac{\tau-1}{m-\tau})^m$ for $m=3,4,...,k-1$. The latter quantity reaches its maximum over $m=3,4,...,k-1$ at $m$ near $m_0:=\sqrt{k(\tau-1)/\e}$. Using a Gaussian approximation of ${k\choose m}(\frac{\tau-1}{m-\tau})^m$ near $m_0$, see Section~\ref{appA} for details, we get the following result.

\begin{thm}[Asymptotic order of cliques] \label{thm5.5}
In the rank-1 inhomogeneous random graph with weight density \eqref{5.2}
and connection probability \eqref{eq:phh}, the average number of cliques of size $k \ge 3$ satisfies
\beq \label{5.18}
A_k(n)=O\,\Bigl[\Bigl(\!\ba{c} n \\ k \ea\!\Bigr)\,\cutoff^{k(1-\tau)}\,{\rm e}^{2\sqrt{k(\tau-1)/e}}\Bigr]. 
\eq 
\end{thm}

Theorem \ref{thm5.5} shows that the expression in square brackets in \eqref{5.17} grows subexponentially in $k$, which is relevant for large cluster sizes $k$. 

\subsection{Precise asymptotics for cycles} \label{subsec5.2}
Using \eqref{defcycle}, the average number $C_k(n)$ of $k$-cycles with $k\geq3$ equals
\begin{align} \label{5.21}
C_k(n) & =  \dfrac{1}{2k}\,\Bigl(\!\ba{c} n \\ k \ea\!\Bigr)\,k!\,\il_1^{\infty}\cdots\il_1^{\infty}\,\rho(h_1)\cdots\rho(h_k) \nonumber\\
&  \times\:f\Bigl(\frac{h_1h_2}{\cutoff^2}\Bigr)\,f\Bigl(\frac{h_2h_3}{\cutoff^2}\Bigr) \cdots f\Bigl(\frac{h_{k-1}h_k}{\cutoff^2}\Bigr)\,f\Bigl(\frac{h_kh_1}{\cutoff^2}\Bigr) dh_1\,dh_2\cdots dh_{k-1}\,dh_k. 
\end{align}
The integral in (\ref{5.21}) is the probability that a particular set of vertices $\{i_1,i_2,...,i_{k-1},i_k\}$ constitutes a $k$-cycle $i_1\pr i_2\pr \cdots\pr i_{k-1}\pr i_k\pr i_1$. 
In (\ref{5.21}), we now allow a general $f$ from the class introduced in \cite{hofstad2017b}.
The main result for $C_k(n)$ reads as follows.

\begin{thm}[Precise asymptotics for cycles] \label{thm5.6}
In the rank-1 inhomogeneous random graph with weight density \eqref{5.2}
and general class of connection probabilities \eqref{eq:phhgeneral}, the average number of cycles of size $k \ge 3$ scales as 
\beq \label{5.22}
C_k(n)
\sim \cutoff^{k(1-\tau)}(\tau-1)^k\,\frac{1}{2k}\,\Bigl(\!\ba{c} n \\ k \ea\!\Bigr)\,k!\,\il_{-A}^A\cdots \il_{-A}^A\,F(\C \bft)\,d\bft, 
\eq
where $A={\rm log}\,\cutoff$,
\beq \label{5.23}
F(u_1,...,u_k)=j(u_1)\,j(u_2)\cdots j(u_k), ~~~~~~u_1,...,u_k\in\dR, 
\eq
with
\beq \label{5.24}
j(u)={\rm e}^{-\frac12(\tau-1)u}\,f({\rm e}^u), ~~~~~~u\in\dR, 
\eq
and $\C$ and $\bft$ are the circulant $k\times k$-matrix and $k$-vector
\beq \label{5.25}
\C=\left[\ba{lllllllllll}
1 & 1 & 0 & 0 & & & & & 0 & 0 & 0 \\
0 & 1 & 1 & 0 & & & & & 0 & 0 & 0 \\
& & & & & & & & & & \\
& & & & & & & & & & \\
& & & & & & & & & & \\
& & & & & & & & & & \\
& & & & & & & & & & \\
& & & & & & & & & & \\
0 & 0 & 0 & 0 & & & & & 1 & 1 & 0 \\
0 & 0 & 0 & 0 & & & & & 0 & 1 & 1 \\
1 & 0 & 0 & 0 & & & & & 0 & 0 & 1
\ea\right], ~~~~~~\bft=\left[\ba{c}
t_1 \\
t_2 \\
\mbox{} \\
\mbox{} \\
\mbox{} \\
\mbox{} \\
\mbox{} \\
\mbox{} \\
t_{k-2} \\
t_{k-1} \\
t_k
\ea\right]. 
\eq
\end{thm}
\mbox{} \\

The proof of Theorem \ref{thm5.6}, detailed in Section~\ref{appB}, uses again the substitution $v_i=h_i/\cutoff$, yielding a factor $\cutoff^{k(1-\tau)}$ outside the integral in (\ref{5.21}) with an integrand of the form \begin{equation}g(v_1v_2)\,g(v_2v_3)\cdots g(v_{k-1}v_k)\,g(v_kv_1).\end{equation} For such an integrand, it is natural to further substitute $t_i={\rm log}\,v_i$, linearizing the arguments of the $g$-functions with the linear algebra of circulant matrices presenting itself.

As to evaluating the remaining integral in Theorem \ref{thm5.6}, we note that
\beq \label{5.26}
{\rm det}(\C)=\left\{\ba{llll}
2,  & ~k~{\rm odd}, \\
0,  & ~k~{\rm even}, 
\ea\right.
\eq
while, due to the properties of $f$ in (\ref{5.5}), the function $F$ is absolutely integrable over $\bfu\in\dR^k$, see (\ref{5.23})--(\ref{5.24}). Thus, for odd $k$, the integral in (\ref{5.22}) remains finite as $A={\rm log}\,\cutoff\pr\infty$, and basic calculus yields the following result.
\begin{thm}[Specific asymptotics for odd cycles] \label{thm5.7}
In the rank-1 inhomogeneous random graph with weight density \eqref{5.2}
and general class of connection probabilities \eqref{eq:phh}, the average number of cycles of odd size $k \ge 3$ scales as 
\beq \label{5.27}
C_k(n)\sim \cutoff^{k(1-\tau)}(\tau-1)^k\,\frac{1}{4k}\, \Bigl(\!\ba{c} n \\ k \ea\!\Bigr)\,k!\,\left( \il_0^{\infty}\,x^{-\frac12(\tau+1)}\,f(x)\,dx\right)^k. 
\eq
\end{thm}
\mbox{}

The integral in (\ref{5.27}) is finite, and can be evaluated in closed form for the standard choices of $f$:
\begin{align} \label{A41}
\il_0^{\infty}\,x^{-\frac12(\tau+1)}\,\min\{1,x\}\,dx&=\frac{4}{(3-\tau)(\tau-1)}, \\
 \label{A42}
\il_0^{\infty}\,x^{-\frac12(\tau+1)}\,\frac{x}{1+x}\,dx&=\frac{\pi}{\sin(\tfrac{\pi}{2}\,(\tau-1))}, \\\label{A43}
\il_0^{\infty}\,x^{-\frac12(\tau+1)}(1-{\rm e}^{-x})\,dx&=\frac{\Gamma(\tfrac12(3-\tau))}{(\tfrac12(\tau-1))}. 
\end{align}

\begin{remark}
Choose $f(x)=\min\{1,x\}$ and use Stirling's formula to approximate 
$
{n\choose k}k!\sim n^k {\rm e}^{-k^2/2n}
$
with validity range $k=o(n^{2/3})$. Then \eqref{5.27} gives
\beq \label{5.27b}
C_k(n)\sim  n^{k(3-\tau)/2}\mu^{k(1-\tau)/2} \Big(\frac{4}{3-\tau}\Big)^k\frac{1}{4k}{\rm e}^{-k^2/2n}.
\eq
\end{remark}

The situation for even $k$ is more delicate, due to singularity of the matrix $\C$. In Section~\ref{appB}, the spectral structure of $\C$ is examined. It appears that $\C$ is diagonizable, with the $k$-DFT vectors as eigenvectors, and precisely one eigenvalue 0, viz.\ the one corresponding to the eigenvector
\beq \label{5.28}
\bfc=\frac{1}{\sqrt{k}}\,({-}1,1,{-}1,1,...,{-}1,1)^T. 
\eq
The integration over $\bft$ in (\ref{5.22}) should now be split into a 1-dimensional integration in the direction of $\bfc$ and a $(k-1)$-dimensional integration over the orthogonal complement $L$ of $\bfc$. The integration over $L$ yields a finite result as $A\pr\infty$, due to invertibility of $\C$ on $L$ and absolute integrability of $F(\bfu)$, $\bfu\in\dR^k$. The integration in the direction of $\bfc$ yields a factor $2A\,\sqrt{k}$. By appropriately representing the integration over $L$ using a delta function, the integral over $L$ can be given in closed form, see Section~\ref{appB}. The final result is as follows.

\begin{thm}[Specific asymptotics for even cycles] \label{thm5.8}
In the rank-1 inhomogeneous random graph with weight density \eqref{5.2}
and general class of connection probabilities \eqref{eq:phhgeneral}, the average number of cycles of even size $k \ge 4$ scales as 
\beq \label{5.29}
C_k(n)\sim (\cutoff^{k(1-\tau)}\,{\rm log}\,\cutoff)(\tau-1)^k\,\frac1k\,\Bigl(\!\ba{c} n \\ k \ea\!\Bigr)\,k! \,\il_{-\infty}^{\infty}\,|J(v)|^k\,dv, 
\eq
where
\beq \label{5.30}
J(v)=\il_0^{\infty}\,x^{-2\pi iv-\frac12(\tau+1)}\,f(x)\,dx, ~~~~~~v\in\dR. 
\eq
\end{thm}
\mbox{}

The integral in (\ref{5.30}) converges, and again gives closed-form expressions:
\begin{align} \label{A62}
\il_0^{\infty}\,x^{-2\pi iv-\frac12(\tau+1)}\,\min\{1,x\}\,dx&=\frac{4}{(3-\tau-4\pi iv)(\tau-1+4\pi iv)}, \\
 \label{A63}
\il_0^{\infty}\,x^{-2\pi iv-\frac12(\tau+1)}\,\frac{x}{1+x}\,dx&= \frac{\pi}{\sin(\tfrac12(\tau-1)+2\pi iv)\,\pi}, \\
\label{A64}
\il_0^{\infty}\,x^{-2\pi iv-\frac12(\tau+1)}\,(1-{\rm e}^{-x})\,dx&=\frac{\Gamma(\tfrac12(3-\tau)-2\pi iv)} {\tfrac12(\tau-1)+2\pi iv}. 
\end{align}

\begin{remark}
Choose again $f(x)=\min\{1,x\}$. Noting that $|J(v)|\leq J(0)=4/(3-\tau)(\tau-1)$, we have
\beq
\il_{-\infty}^{\infty}\,|J(v)|^k\,dv \asymp \Big(\frac{4}{(3-\tau)(\tau-1)}\Big)^k\frac{1}{\sqrt{k}}. 
\eq
Thus we get for even $k$ a similar expression for $C_k(n)$ as 
\eqref{5.27b}, except for an additional factor $\log (\mu n)/\sqrt{k}$. This agrees with the observation in Remark 
\ref{remm1} that for even $k$, $C_k(n)$ grows faster than $A_k(n)$.

\end{remark}



\section{Remaining proofs for cliques}  \label{appA}
To reduce notational complexity, we replace the upper integration limits $\infty$ in 
\eqref{5.12} and \eqref{5.21} by $\cutoff^2$, at the expense of relative errors $o(1)$ as $n\to\infty$. \\

\noindent {\bf Proof of Proposition~\ref{prop5.1}.}~~With the substitution
\beq \label{A1}
v_i=\frac{h_i}{\cutoff}\in[w,W]~;~~~~~~w=\frac{1}{\cutoff}\,,~~W=\cutoff, 
\eq
we get for $m=1,...,k-1$
\begin{align} \label{A2}
 I_m\sim \cutoff^{k(1-\tau)}(\tau-1)^k\,
&\begin{array}[t]{c} \underbrace{\il_w^1\cdots\il_w^1} \\ m \end{array}
\begin{array}[t]{c} \underbrace{\il_1^W\cdots\il_1^W} \\ k-m\end{array}
\,v_1^{-\tau}\cdots v_m^{-\tau}v_{m+1}^{-\tau}\cdots v_k^{-\tau} \nonumber\\
&\times\prod_{1\leq i<j\leq k}\,f(v_iv_j)\,dv_1\cdots dv_mdv_{m+1}\cdots dv_k. 
\end{align}
We have for $w\leq v_1,...,v_m\leq 1\leq v_{m+1},...,v_k\leq W$
\beq \label{A3}
\prod_{1\leq i<j\leq k}\,f(v_iv_j)=\Bigl(\prod_{i=1}^m\,v_i^{m-1}\Bigr)\,\prod_{i\leq i\leq m<j}\,\min\{1,v_iv_j\}
\eq
since $f(x)=\min\{1,x\}$. Therefore
\begin{align}\label{A4}
& \mbox{}  v_1^{-\tau}\cdots v_m^{-\tau}v_{m+1}^{-\tau}\cdots v_k^{-\tau}\,\prod_{1\leq i<j\leq k}\,f(v_iv_j) \nonumber \\
&  =~\Bigl(\prod_{i=1}^m\,v_i^{m-1-\tau}\Bigr)\,\prod_{j=m+1}^k\,\Bigl(v_j^{-\tau}\,\prod_{i=1}^m\,\min \{1,v_iv_j\}\Bigr). 
\end{align}
As a consequence, the integral in (\ref{A2}) over $v_{m+1},...,v_k$ factorizes, and we get
\beq \label{A5}
I_m\sim \cutoff^{k(1-\tau)}(\tau-1)^k\,\il_w^1\cdots\il_w^1\,\Bigl(\prod_{i=1}^m\,v_i^{m-1-\tau}\Bigr) \,F_m^{k-m}(v_1,...,v_m)\,dv_1\cdots dv_m, 
\eq
where
\beq \label{A6}
F_m(v_1,...,v_m)=\il_1^W\,v^{-\tau}\,\prod_{i=1}^m\,\min\{1,v_iv\}\,dv. 
\eq
We now observe that $F_m$ in (\ref{A6}) is a symmetric function of $v_1,...,v_m$. Thus we shall evaluate (\ref{A6}) for increasingly ordered $v_i$.

We have for $w\leq v_1\leq\cdots\leq v_m\leq1$ (so that $1/v_m\leq\cdots\leq 1/v_1\leq 1/w=W$), where we assume $m\geq3$,
\begin{eqnarray} \label{A9}
& \mbox{} & \hspace*{-2mm}F_m(v_1,...,v_m) = \il_1^W\,v^{-\tau}\,\min\{1,v_1v\}\cdots\min\{ 1,v_mv\}\,dv \nonumber \\
& & \hspace*{-2mm}=~ \il_1^{1/v_m}\,v^{-\tau}v_1v\cdots v_mv\,dv+\il_{1/v_m}^{1/v_{m-1}}\,v^{-\tau}v_1v\cdots v_{m-1}v\,dv \nonumber \\
& & \hspace*{5mm}+\cdots +\il_{1/v_2}^{1/v_1}\,v^{-\tau}v_1v\,dv+\il_{1/v_1}^W\,v^{-\tau}\,dv \nonumber \\
& & \hspace*{-2mm}=~\frac{v_1\cdots v_m}{m-\tau+1}\,\Bigl(\Bigl(\frac{1}{v_m}\Bigr)^{m-\tau+1}-1\Bigr)+ \frac{v_1\cdots v_{m-1}}{m-\tau}\,\Bigl(\Big(\frac{1}{v_{m-1}}\Bigr)^{m-\tau}-\Bigl(\frac{1}{v_m} \Bigr)^{m-\tau}\Bigr) \nonumber \\
& & \hspace*{5mm}+\cdots +\frac{v_1}{2-\tau}\,\Bigl(\Bigl(\frac{1}{v_1}\Bigr)^{2-\tau}-\Bigl(\frac{1}{v_2}\Bigr)^{2-\tau}\Bigr) + \frac{1}{1-\tau}\,\Bigl(W^{1-\tau}-\Bigl(\frac{1}{v_1}\Bigr)^{1-\tau}\Bigr). 
\mbox{}
\end{eqnarray}
Letting $W\pr\infty$, so that $W^{1-\tau}\pr0$ since $2<\tau<3$, and using
\beq \label{A10}
\frac{1}{j-\tau+1}-\frac{1}{j-\tau}=\frac{-1}{(j-\tau+1)(j-\tau)}, 
\eq
we then get upon some further rewriting
\begin{eqnarray} \label{A11}
& \mbox{} & F_m(v_1,...,v_m) \sim v_1v_2^{\tau-2}\,\Bigl(\frac{1}{(3-\tau)(\tau-2)}-\frac{(v_1/v_2)^{\tau-2}}{(\tau-2)(\tau-1)} \nonumber \\
& & -~\sum_{j=3}^m\,\frac{(v_2/v_j)^{3-\tau}}{(j+1-\tau)(j-\tau)}\:\frac{v_3\cdots v_{j-1}}{v_j^{j-3}}- \frac{v_2^{3-\tau}}{m-\tau+1}\,v_3\cdots v_m\Bigr).  
\end{eqnarray}
For $m=1,2\,$, we get
\beq \label{A12}
F_1(v_1)\sim v_1\Bigl(\frac{1}{\tau-2}-\frac{v_1^{\tau-2}}{(\tau-1)(\tau-2)}\Bigr), 
\eq
\beq \label{A13}
F_2(v_1,v_2)\sim v_1v_2^{\tau-2}\Bigl(\frac{1}{(3-\tau)(\tau-2)}- \frac{(v_1/v_2)^{\tau-2}}{(\tau-2)(\tau-1)}-\frac{v_2^{3-\tau}}{3-\tau}\Bigr), 
\eq
where it is assumed that $w\leq v_1\leq1$ and $w\leq v_1\leq v_2\leq1$ in the respective cases.

To summarize, we have from (\ref{A5}) and symmetry of $F_m$ in (\ref{A6})
\begin{eqnarray} \label{A16}
& \mbox{} & I_m \sim \cutoff^{k(1-\tau)}(\tau-1)^k\,m! \nonumber \\ 
& & \times~~~\il\cdots\il_{\hspace*{-1.2cm}0<v_1\leq\cdots\leq v_m\leq1}\Bigl(\prod_{i=1}^m\,v_i^{m-1-\tau}\Bigr)\,F_m^{k-m}(v_1,...,v_m)\,dv_1\cdots dv_m, 
\end{eqnarray}
where $F_m$ is given for $m=1,...,k-1$ by (\ref{A11})--(\ref{A13}). Here we have replaced the lower integration limit $w$ of $v_1$ by $0$ at the expense of a relative error $o(1)$ as $n\to\infty$. 

To complete the proof of Proposition~\ref{prop5.1}, we substitute
\beq \label{A18}
t_1=\frac{v_1}{v_2},...,t_{m-1}=\frac{v_{m-1}}{v_m}\,,t_m=v_m\in(0,1]. 
\eq
From
\beq \label{A20}
v_m=t_m,v_{m-1}=t_{m-1}t_m,...,v_1=t_1t_2\cdots t_m, 
\eq
we have
\beq \label{A22}
F_m(v_1,...,v_m)\sim\frac{1}{\tau-1}\,\Phi_m(t_1,...,t_m), 
\eq
with $\Phi_m$ given in \eqref{5.13})--\eqref{5.15}. Furthermore, from (\ref{A20})
\beq \label{A24}\\
{\rm det}\Bigl(\frac{\partial v_r}{\partial t_i}\Bigr)_{i,r=1,...,m}=\prod_{i=1}^m\, t_i^{i-1}. 
\eq
Finally, rewriting the products in (\ref{A16}) using (\ref{A20}), we obtain \eqref{5.12}--\eqref{5.13a}. \\ \\
{\bf Proof of Proposition~\ref{prop5.3}.}~~The cases $m=1,2$ can be dealt with directly using \eqref{5.13}--\eqref{5.14}. We assume now $m=3,4,...,k-1$. \\

(i) We have $\Phi_m(0,...,0)=1$ at once from \eqref{5.15}, and
\begin{align} \label{A26}
\Phi_m(1,...,1) & =  (\tau-1)\,\Bigl(\frac{1}{(3-\tau)(\tau-2)}- \frac{1}{(\tau-1)(\tau-2)} \nonumber \\ 
&  -~\sum_{j=3}^m\,\frac{1}{(j+1-\tau)(j-\tau)}-\frac{1}{m+1-\tau}\Bigr)=1, 
\end{align}
where we have used (\ref{A10}). \\
\mbox{}

(ii) and (iii) We write for $0\leq t_1,...,t_m\leq1$
\beq \label{A27}
\Phi_m(t_1,...,t_m)=t_1t_2^{\tau-1}\cdots t_m^{\tau-1}\,\Psi(t_1,...,t_m), 
\eq
where
\begin{align} \label{A28}
\hspace*{-4mm}\Psi_m(t_1,...,t_m) & =  (\tau-1)\,\Bigl(\frac{1}{(3-\tau)(\tau-2)}- \frac{t_1^{\tau-2}}{(\tau-1)(\tau-2)} \nonumber \\ 
&  -~\sum_{j=3}^m\,\frac{t_2^{3-\tau}t_3^{4-\tau}\cdots t_{j-1}^{j-\tau}}{(j+1-\tau)(j-\tau)}
- \frac{t_2^{3-\tau}t_3^{4-\tau}\cdots t_m^{m+1-\tau}}{m+1-\tau}\Bigr). 
\end{align}
Since $\tau\in(2,3)$, we see from (\ref{A26}) that $\Psi_m(t_1,...,t_m)\geq1$, with equality only when $t_1=\cdots =t_m=1$.

We consider the cases $i=1$ and $i=2,...,m$ separately. We have from (\ref{A27})--(\ref{A28})
\begin{align} \label{A29}
\frac{\partial\Phi_m}{\partial t_1}\,(t_1,...,t_m) & =  t_2^{\tau-1}\cdots t_m^{\tau-1}\, \Psi_m(t_1,...,t_m) \nonumber \\ 
&  +~t_1t_2^{\tau-1}\cdots t_m^{\tau-1}\times ({-}t_1^{\tau-3}) \nonumber \\ 
& \geq  t_2^{\tau-1}\cdots t_m^{\tau-1}(1-t_1^{\tau-2})\geq0, 
\end{align}
with equality if $t_1=\cdots =t_m=1$. Next, let $i=2,...,m\,$. We have for $0\leq t_1,...,t_m\leq1$ from (\ref{A27})--(\ref{A28})
\begin{eqnarray} \label{A30}
& \mbox{} & \frac{\partial \Phi_m}{\partial t_i}\,(t_1,...,t_m)=(\tau-1)t_i^{-1}t_1t_2^{\tau-1}\cdots t_m^{\tau-1}\,\Psi(t_1,...,t_m) \nonumber \\ 
& & \hspace*{5mm}-~(\tau-1)t_1t_2^{\tau-1}\cdots t_m^{\tau-1}\,\Bigl(\sum_{j=i+1}^m\,(i+1-\tau)t_i^{-1}\, \frac{t_2^{3-\tau}\cdots t_{j-1}^{j-\tau}} {(j+1-\tau)(j-\tau)} \nonumber \\ 
& & \hspace*{5cm}+~(i+1-\tau)t_i^{-1}\,\frac{t_2^{3-\tau}\cdots t_m^{m+1-\tau}}{m+1-\tau}\Bigr) \nonumber \\ 
& & \geq~(\tau-1)t_i^{-1}t_1t_2^{\tau-1}\cdots t_m^{\tau-1}\Bigl(1-\sum_{j=i+1}^m \,\frac{i+1-\tau}{(j+1-\tau)(j-\tau)}-\frac{i+1-\tau}{m+1-\tau}\Bigr), \nonumber \\
\mbox{}
\end{eqnarray}
and the final member of (\ref{A30}) equals 0 by (\ref{A10}). There is equality in (\ref{A30}) for $t_1=\cdots t_m=1$. 
(iv) is shown in a similar fashion as (iii).\\ \\
{\bf Proof of Proposition~\ref{prop5.4}.}~~(i) Let $0\leq t_1,...,t_m\leq1$. Since $0\leq\Phi_m(t_1,...,t_m)\leq1$, we see that $\Phi^{k-m}(t_1,...,t_m)$ decreases in $k$, and so does $J_m$. 
(ii) Let $m=3,4,...,k-1$. Since $0\leq\Phi_m(t_1,...,t_m)\leq1$, we have
\beq \label{A22a}
J_m\leq\il_0^1...\il_0^1\:\prod_{i=1}^m\,t_i^{i(m-\tau)-1}\,dt_1\cdots dt_m=\frac{1}{m!}\,\Bigl(\frac{1}{m-\tau}\Bigr)^m. 
\eq 
\mbox{} \\
{\bf Proof of Theorem~\ref{thm5.5}.}~~We must bound the quantity in $[~~]$ at the right-hand side of \eqref{5.17}. By Proposition~\ref{prop5.4} (i), we have that $J_1$ and $J_2$ are bounded, and so 
\beq \label{A23a}
1+\sum_{m=1}^{k-1}\,\Bigl(\!\ba{c} k \\ m \ea\!\Bigr)\,(\tau-1)^m\,m!\,J_m+\Bigl(\frac{\tau-1}{k-\tau}\Bigr)^k=\sum_{m=3}^{k-1}\,\Bigl(\!\ba{c} k \\ m \ea\!\Bigr)\,(\tau-1)^m\,m!\,J_m+O(k^2). 
\eq 
By Proposition~\ref{prop5.4} (ii), we have
\beq \label{A24a}
\sum_{m=3}^{k-1}\,\Bigl(\!\ba{c} k \\ m \ea\!\Bigr)\,(\tau-1)^m\,m!\,J_m\leq\sum_{m=3}^{k-1}\,t_m\,;~~t_m=\Bigl(\!\ba{c} k \\ m \ea\!\Bigr)\,\Bigl(\frac{\tau-1}{m-\tau}\Bigr)^m. 
\eq 
The ratio $t_m/t_{m+1}$ of two consecutive terms equals
\beq \label{A25a}
\Bigl(1+\frac{1}{m-\tau}\Bigr)^m\:\frac{(m+1-\tau)(m+1)}{(\tau-1)(k-m)}\sim \frac{\e}{\tau-1}~\frac{(m+1-\tau)(m+1)}{k-m}, 
\eq 
and this exceeds 1 from $m\sim m_0:=\sqrt{k(\tau-1)/\e}$ onwards. Thus, the largest terms in $\sum_{m=3}^{k-1}\,t_m$ occur for $m$ of the order $\sqrt{k}$. With $m=O(\sqrt{k})$, we have
\begin{align} \label{A26a}
\Bigl(\!\ba{c} k \\ m \ea\!\Bigr) & =  \frac{k^m}{m!}\,\exp\,\Bigl(\sum_{j=0}^{m-1}\,{\rm log}\Bigl(1-\frac{j}{k}\Bigr)\Bigr) \nonumber \\ 
& =  \frac{k^m}{m!}\,\exp\,\Bigl({-}\,\frac{m(m-1)}{2k}+O\Bigl(\frac{m^3}{k^2}\Bigr)\Bigr)=O\Bigl(\frac{k^m}{m!}\Bigr). 
\end{align}
Then using Stirling's formula, $m!=m^m\,{\rm e}^{-m}\,\sqrt{2\pi m}\,(1+O(1/m))$, we find that
\beq \label{A27a}
\Bigl(\!\ba{c} k \\ m \ea\!\Bigr)\,\Bigl(\frac{\tau-1}{m-\tau}\Bigr)^m=O\Bigl[\Bigl(\frac{k\e(\tau-1)}{m(m-\tau)}\Bigr)^m\,\frac{1}{\sqrt{2\pi m}}\Bigr]. 
\eq 
We aim at a Gaussian approximation of the dominant factor $[k\e(\tau-1)/m(m-\tau)]^m$ at the right-hand side of (\ref{A27a}). We have 
\begin{align} \label{A28a}
\frac{d}{dm}\,{\rm log}\Bigl(\frac{k\e(\tau-1)}{m(m-\tau)}\Bigr)^m & =  {\rm log}\Bigl(\frac{k(\tau-1)}{\e m^2}\Bigr)-{\rm log}\Bigl(1-\frac{\tau}{m}\Bigr)-\frac{\tau}{m-\tau} \nonumber \\ 
& =  {\rm log}\Bigl(\frac{k(\tau-1)}{\e m^2}\Bigr)+O\Bigl(\frac{1}{m^2}\Bigr). 
\end{align}
The leading term at the right-hand side of (\ref{A28a}) vanishes at $m=m_0=\sqrt{k(\tau-1)/\e}$. At $m=m_0$, we evaluate
\beq \label{A29a}
\Bigl(\frac{k\e(\tau-1)}{m(m-\tau)}\Bigr)^m={\rm e}^{2m_0+\tau}\,\Bigl(1+O\Bigl(\frac{1}{m_0}\Bigr)\Bigr), 
\eq
\beq \label{A30a}
\Bigl(\frac{d}{dm}\Bigr)^2\,\Bigl(\frac{k\e(\tau-1)}{m(m-\tau)}\Bigr)^m={-}\,\frac{2}{m_0}+O\Bigl(\frac{1}{m_0^2}\Bigr). 
\eq 
Thus, we find the Gaussian approximation
\beq \label{A31a}
\Bigl(\frac{k\e(\tau-1)}{m(m-\tau)}\Bigr)^m\sim \exp\,\Bigl(2m_0+\tau-\frac{1}{m_0}\,(m-m_0)^2\Bigr)
\eq 
for $m$ near $m_0$ (validity range: $|m-m_0|=o(m_0^{2/3})$). Then, from (\ref{A23a}), (\ref{A24a}), (\ref{A27a}) and (\ref{A31a}), we get
\begin{eqnarray} \label{A32a}
& \mbox{} & 1+\sum_{m=1}^{k-1}\,\Bigl(\!\ba{c} k \\ m \ea\!\Bigr)\,(\tau-1)^m\,m!\,J_m+\Bigl(\frac{\tau-1}{k-\tau}\Bigr)^k \nonumber \\ 
& & =~O\,\Bigl[\frac{1}{\sqrt{m_0}}\,{\rm e}^{2m_0}\,\sum_{m={-}\infty}^{\infty}\,\exp\,\Bigl({-}\,\frac{1}{m_0}\,(m-m_0)^2\Bigr)\Bigr]+O(k^2) \nonumber \\ 
& & =~O({\rm e}^{2m_0}), 
\end{eqnarray}
as required.

\section{Remaining proofs for cycles}  \label{appB}
We replace  the upper integration limits $\infty$ in 
\eqref{5.21} by $\cutoff^2$, as in Section \ref{appA}. \\

\noindent {\bf Proof of Theorem~\ref{thm5.6}.}~~After the basic substitution in (\ref{A1}), we get
\begin{align} \label{A31}
C_k(n)   \sim \  & \cutoff^{k(1-\tau)}(\tau-1)^k\,\frac{1}{2k}\,
\Bigl(\!\ba{c} n \\ k \ea\!\Bigr)\, k! \nonumber \\ 
&  \times~\il_w^W\cdots \il_w^W\,h(v_1,v_2,...,v_k)\,dv_1\cdots dv_k, 
\end{align}
where
\begin{align} \label{A32}
h(v_1,...,v_k) & =  v_1^{-\tau}\cdots v_k^{-\tau}\,f(v_1v_2)\,f(v_2v_3)\cdots f(v_{k-1}v_k)\,f(v_kv_1) \nonumber \\ 
& =  g(v_1v_2)\,g(v_2v_3)\cdots g(v_{k-1}v_k)\,g(v_kv_1), 
\end{align}
with $g(x)=x^{-\tau/2}\,f(x)$. The substitution
\beq \label{A33}
v_i={\rm e}^{t_i}\,,\:\:{-}A\leq t_i\leq A\,;~~~dv_i={\rm e}^{t_i}dt_i\,;~~~A={-}{\rm log}\,w={\rm log}\,W={\rm log}\,\cutoff
\eq
then yields
\begin{eqnarray} \label{A34}
& \mbox{} & C_k(n)\sim \cutoff^{k(1-\tau)}(\tau-1)^k\,\frac{1}{2k}\,
\Bigl(\!\ba{c} n \\ k \ea\!\Bigr)\,k! \nonumber \\ 
& & \times~\il_{-A}^A\cdots\il_{-A}^A\,j(t_1+t_2)j(t_2+t_3)\cdots j(t_{k-1}+t_k)j(t_k+t_1)dt_1\cdots dt_k, \nonumber \\
\mbox{}
\end{eqnarray}
where
\beq \label{A35}
j(u)={\rm e}^{\frac12 u}\,g({\rm e}^u)={\rm e}^{-\frac12(\tau-1)u}\,f({\rm e}^u), ~~~~~~u\in\dR, 
\eq
and Theorem~\ref{thm5.6} follows. \\ \\
{\bf Proof of Theorem~\ref{thm5.7}.}~~The formula \eqref{5.26} for ${\rm det}(\C)$ follows from basic matrix operations with $\C$ in \eqref{5.25}. Hence, $\C$ is non-singular when $k$ is odd. Next, from \eqref{5.5} and (\ref{A35}), we have
\beq \label{A36}
j(u)=O({\rm e}^{\frac12(3-\tau)u})\,,~~u<0~;~~~~~~j(u)=O({\rm e}^{-\frac12(\tau-1)u})\,,~~u>0, 
\eq
and so $j(u)$ has exponential decay as $|u|\pr\infty$. Therefore, $F(\bfu)$ in \eqref{5.23} is absolutely integrable over $\dR^k$, and by the substitution $\bfu=\C\bft$, with ${\rm det}(\C)=2$, we get
\beq \label{A37}
\il_{-A}^A\cdots\il_{-A}^A\,F(\C\bft)\,d\bft=\frac12\,\il\cdots\il_{\hspace*{-1.2cm}R(A)}\,F(\bfu)\,d\bfu
\eq
with integration range $R(A)=\C([{-}A,A]^k)$. By non-singularity of $\C$, there is a $\delta>0$ such that
\beq \label{A38}
R(A)\supset [{-}\delta A,\delta A]^k, ~~~~~~A>0, 
\eq
and so we get
\begin{align} \label{A39}
\lim_{A\pr\infty}\,\il_{-A}^A\cdots\il_{-A}^A\,F(\C\bft)\,d\bft & =  \frac12\,\il_{-\infty}^{\infty} \cdots\il_{-\infty}^{\infty}\,F(\bfu)\,d\bfu \nonumber \\ 
& =  \frac12\left(\,\il_{-\infty}^{\infty}\,j(u)\,du\right)^k, 
\end{align}
where we have used the definition of $F$ in \eqref{5.23}. Finally, by \eqref{5.24} and the substitution $x={\rm e}^u\in(0,\infty)$, we get
\beq \label{A40}
\il_{-\infty}^{\infty}\,j(u)\,du=\il_0^{\infty}\,x^{-\frac12(\tau+1)}\,f(x)\,dx, 
\eq
and this is finite because of \eqref{5.5} and $2<\tau<3$. \\ \\
{\bf Proof of Theorem~\ref{thm5.8}.}~~Let $k$ be even. From the theory of circulant matrices, we have that $\C$ is diagonizable,
\beq \label{A44}
\C=\sum_{m=1}^k\,\lambda_m\,\bfd_m\,\bfd_m^H, 
\eq
where for $m=1,...,k$
\beq \label{A45}
\lambda_m=1+{\rm e}^{2\pi im/k}, ~~~~~~\bfd_m=\Bigl(\frac{1}{\sqrt{k}}\,{\rm e}^{2\pi imr/k}\Bigr)_{r=1,...,k}
\eq
are the eigenvalues and eigenvectors of $\C$. With $k=2j$, we have
\beq \label{A46}
\lambda_j=\lambda_{\frac12 k}=0~;~~~~~~\lambda_m\neq0\,,~~m=1,...,k\,,~~m\neq j. 
\eq
Let
\beq \label{A47}
\bfc=\bfd_j=\frac{1}{\sqrt{k}}\,({-}1,1,...,{-}1,1)^T, ~~~~~~L=\langle \bfc\rangle^{\perp}
\eq
be the eigenvector of $\C$ corresponding to $\lambda_j=0$ and let $L$ be its orthogonal complement. It follows from (\ref{A44}) that $\C$ maps $L$ linearly and injectively onto itself.

For $\bft\in\dR^k$, we write
\beq \label{A48}
\bft=\bfw+a\bfc, ~~~~~~\bfw\in L\,,~~a\in\dR. 
\eq
Then $\C\bft=\C\bfw$, and
\beq \label{A49}
\il_{-A}^A\cdots\il_{-A}^A\,F(\C\bft)\,d\bft=
\ba[t]{c} \dil_a~\dil_{\bfw\in L} \\
\mbox{\footnotesize $\bfw+a\bfc\in[{-}A,A]^k$}
\ea
\,F(\C\bfw)\,d\bfw\,da. 
\eq
Observe that $A\,\sqrt{k}\,\bfc=({-}A,A,...,{-}A,A)^T$ is a corner point of $[{-}A,A]^k$. Let $\eps\in(0,1)$ and assume that $a\in\dR$, $|a|<(1-\eps)\,A\,\sqrt{k}$. Then
\beq \label{A50}
|(a\bfc)_r|<(1-\eps)\,A, ~~~~~~r=1,...,k, 
\eq
and so
\beq \label{A51}
|a|<(1-\eps)\,A\,\sqrt{k}~~\&~~\bfw\in[{-}\eps A,\eps A]^k\Rightarrow \bfw+a\bfc\in[{-}A,A]^k. 
\eq
The function $F(\bfu)$ is absolutely integrable over $\bfu\in\dR^k$ and $\C$ is boundedly invertible on $L$. Therefore, $\int_{\bfw\in L}\,F(\C\bfw)\,d\bfw$ is finite. It follows from (\ref{A51}) that for any $\eps\in (0,1)$
\beq \label{A52}
\lim_{A\pr\infty}\,
\ba[t]{c} \dil_{\bfw\in L} \\
\mbox{\footnotesize $\bfw+a\bfc\in[{-}A,A]^k$}
\ea
\,F(\C\bfw)\,d\bfw=\il_{\bfw\in L}\,F(\C\bfw)\,d\bfw
\eq
uniformly in $a\in\dR$, $|a|<(1-\eps)\,A\,\sqrt{k}$. Therefore, from (\ref{A49}), as $A\pr\infty$
\beq \label{A53}
\il_{-A}^A\cdots\il_{-A}^A\,F(\C\bft)\,d\bft=2A\,\sqrt{k}\,\il_{\bfw\in L}\,F(\C\bfw)\,d\bfw\,(1+o(1)). 
\eq

There remains to be computed $\int_{\bfw \in L}\,F(\C\bfw)\,d\bfw$. The mapping $\C:L\pr L$ is invertible, and we have from (\ref{A44})--\eqref{A45}
\beq \label{A54}
{\rm det}(\C:L\pr L)=\prod_{m=1,m\neq j}^k\,\lambda_m=\prod_{m=1,m\neq j}^k\,(1+{\rm e}^{2\pi im/k})=k. 
\eq
Thus we have
\beq \label{A55}
\il_{\bfw\in L}\,F(\C\bfw)\,d\bfw=\frac1k\,\il_{\bfu\in L}\,F(\bfu)\,d\bfu. 
\eq
We represent the condition $\bfu\in L$, i.e., $\bfu^T\bfc=0$ with $\bfc$ the vector in (\ref{A47}) having unit Euclidean length, as
\beq \label{A56}
\delta(\bfu^T\bfc=0)=\il_{-\infty}^{\infty}\,{\rm e}^{2\pi is\bfu^T\bfc}\,ds, ~~~~~\bfu\in\dR^k. 
\eq
Hence
\beq \label{A57}
\il_{\bfu\in L}\,F(\bfu)\,d\bfu=\il_{-\infty}^{\infty}~\il_{\bfu\in\dR^k}\,{\rm e}^{2\pi is\bfu^T\bfc}\, F(\bfu)\,d\bfu\,ds. 
\eq
By \eqref{5.23} and (\ref{A47}), we have
\beq \label{A58}
{\rm e}^{2\pi is\bfu^T\bfc}\,F(\bfu)=\prod_{r=1}^j\,[{\rm e}^{-2\pi isu_{2r-1}/\sqrt{k}} \,j(u_{2r-1})] [{\rm e}^{2\pi isu_{2r}/\sqrt{k}}\,j(u_{2r})]. 
\eq
Hence, the integral over $\bfu$ in (\ref{A57}) factorizes, and we get
\begin{align} \label{A59}
\il_{\bfu\in L}\,F(\bfu)du & = 
\il_{-\infty}^{\infty}\,\Big[\,\il_{-\infty}^{\infty}\,{\rm e}^{-2\pi isu/\sqrt{k}}\,j(u)du\, \il_{-\infty}^{\infty}\,{\rm e}^{2\pi isu/\sqrt{k}}\,j(u)du\Big]^jds \nonumber \\ 
& =  \il_{-\infty}^{\infty}\,|J(s/\sqrt{k})|^{2j}\,ds=\sqrt{k}\,\il_{-\infty}^{\infty}\, |J(v)|^k\,dv, 
\end{align}
where
\beq \label{A60}
J(v)=\il_{-\infty}^{\infty}\,{\rm e}^{-2\pi iuv}\,j(u)\,du=\il_0^{\infty}\,x^{-2\pi iv-\frac12(\tau+1)}\, f(x)\,dx
\eq
is the Fourier transform of $j$ in \eqref{5.24}.

Returning to \eqref{5.22}, we then get from (\ref{A53}), (\ref{A55}) and (\ref{A59})
\begin{align} \label{A61}
C_k(n) & \sim  \cutoff^{k(1-\tau)}(\tau-1)^k\,\frac{1}{2k}\,
\Bigl(\!\ba{c} n \\ k \ea\!\Bigr)\,k! \nonumber \\ 
&  \times~2A\,\sqrt{k}\cdot\frac1k\cdot\sqrt{k}\,\il_{-\infty}^{\infty}\,|J(v)|^k\,dv, 
\end{align}
and this yields Theorem~\ref{thm5.8} since $A={\rm log}\,\cutoff$.

\subsection*{Acknowledgement}
This work is supported by NWO Gravitation Networks grant 024.002.003. 
\bibliographystyle{abbrv}   
\bibliography{references2}

\end{document}